\documentclass[11pt]{article}

\usepackage{amsfonts}
\usepackage{amssymb,amsmath,amsthm, enumerate}
\usepackage{latexsym}
\usepackage{fullpage}
\usepackage{hyperref}

\newtheorem{theorem}{Theorem}[section]

\newtheorem{question}[theorem]{Question}

\newtheorem{lemma}[theorem]{Lemma}
\newtheorem{cor}[theorem]{Corollary}
\newtheorem{definition}[theorem]{Definition}

\theoremstyle{definition}

\newcounter{tenumerate}

\def\P{\mathbb{P}}

\newcommand{\one}{\1}

\renewcommand{\epsilon}{\varepsilon}

\newcommand{\1}{\mathbf{1}}

\DeclareMathOperator{\var}{Var}

\newcommand{\R}{{\mathbb R}}
\newcommand{\N}{{\mathbb N}}

\newcommand{\E}{{\mathbb E}}
\newcommand{\remove}[1]{}
\renewcommand{\le}{\leqslant}
\renewcommand{\ge}{\geqslant}
\renewcommand{\leq}{\leqslant}
\renewcommand{\geq}{\geqslant}

\newcommand{\cov}{\mathrm{Cov}}

\def\XXint#1#2#3{{\setbox0=\hbox{$#1{#2#3}{\int}$}
\vcenter{\hbox{$#2#3$}}\kern-.5\wd0}}

\def\EE{\mathbb{E}}

\def\RR{\mathbb{R}}
\def\PP{\mathbb{P}}
\def\Sph{S^{N-1}}

\def \Z{{\Bbb{Z}}}

\def\X{{\bf X}}
\DeclareMathOperator*{\argmax}{arg\,max}

\begin{document}

\title{On multiple peaks and moderate deviations for supremum of Gaussian field}

\author{Jian Ding \thanks{Partially supported by NSF grant DMS-1313596.} \\
University of Chicago\and Ronen Eldan \\ Microsoft Research \and Alex Zhai \\ Stanford University}
\maketitle
\begin{abstract}
We prove two theorems concerning extreme values of general Gaussian fields. Our first theorem concerns with the concept of {\it multiple peaks}. A theorem of Chatterjee states that when a centered Gaussian field admits the so-called {\it superconcentration} property, it typically attains values near its maximum on multiple near-orthogonal sites, known as {\it multiple peaks}. We improve his theorem in two aspects: (i) the number of peaks attained by our bound is of the order $\exp(c / \sigma^2)$ (as opposed to Chatterjee's polynomial bound in $1/\sigma$), where $\sigma$ is the standard deviation of the supremum of the Gaussian field, which is assumed to have variance at most $1$ and (ii) our bound need not assume that the correlations are non-negative. We also prove a similar result based on the superconcentration of the free energy. As primary applications, we infer that for the S-K spin glass model on the $n$-hypercube and directed polymers on $\Z_n^2$, there are polynomially (in $n$) many near-orthogonal sites that achieve values near their respective maxima.

Our second theorem gives an upper bound on moderate deviation for the supremum of a general Gaussian field. While the Gaussian isoperimetric inequality implies a sub-Gaussian concentration bound for the supremum, we show that the exponent in that bound can be improved under the assumption that the expectation of the supremum is of the same order as that of the independent case.
\end{abstract}

\section{Introduction}
A \emph{Gaussian field} (or \emph{Gaussian process}) is a collection $\X = \{X_\alpha, \alpha \in I \}$ of random variables such that every finite subset of this collection is distributed according to a multivariate normal law. The topic of this paper revolves around the behavior of extremal and near-extremal values of Gaussian fields.

Extremal values of Gaussian fields have been intensively studied by a variety of communities spanning probability, statistical physics, and computer science. A cornerstone of the theory is the  Gaussian concentration inequality of Sudakov-Tsirelson \cite{ST74} and Borell \cite{Borell75}, stating that for a (not necessarily centered) Gaussian field $\{X_i: 1\leq i\leq N\}$ with $\sigma^2 = \max_{1\leq i\leq N}\var X_i $, we have
\begin{equation}\label{eq:borel}
\P\Big( \Bigl |\sup_{1\leq i\leq N} X_i-  \E \Big(\sup_{1\leq i\leq N} X_i\Big) \Bigr | \geq z\Big)
\leq \frac{2}{\sqrt{2\pi} \sigma} \int_z^\infty \mathrm{e}^{-\frac{y^2}{2 \sigma^2}} dy
\quad \mbox{ for all } z \geq 0
\end{equation}
(see e.g., \cite[Thm. 7.1, Eq. (7.4)]{Ledoux89}). An immediate consequence of \eqref{eq:borel} is that $\var(\sup_{1\leq i\leq N} X_i) \leq \sigma^2$.
Despite being an extremely general and powerful inequality, it was observed by probabilists and statistical physicists that the bound \eqref{eq:borel} is far from sharp in most canonical examples of Gaussian fields, such as the KPZ universality class \cite{KPZ} and the class of log-correlated Gaussian fields  (see e.g., \cite{Madaule13} and references therein). By being far from sharp, we mean for example that $\var(\sup_{1\leq i\leq N} X_i) \ll \sigma^2$ or that
equation (\ref{eq:borel}) holds with a constant smaller than $\frac{1}{2 \sigma^2}$ in the exponent. The former property is sometimes referred to as \emph{superconcentration} and the latter fits under the umbrella of large deviation estimates. In this paper, we study the structure of Gaussian fields concerning the following two questions related to \eqref{eq:borel}:
\begin{enumerate}[(a)]
\item When \eqref{eq:borel} is not sharp, what extra information can be deduced about the Gaussian field? \label{question-a}

\item Are there some simple and explicit conditions that guarantees an improvement upon \eqref{eq:borel}? \label{question-b}
\end{enumerate}

The rigorous study of Question~\eqref{question-a} in its full generality was pioneered in \cite{Chatterjee08}, where a connection between the so-called \emph{superconcentration}, \emph{chaos} and \emph{multiple peaks} (or \emph{multiple valleys}) phenomena for centered Gaussian fields was established. Multiple peaks is the following phenomenon observed by physicists in many natural settings of Gaussian fields (motivated by the study of energy landscapes of spin glasses): typically there exist many near-orthogonal sites whose values are very close to the global maximum.  This phenomenon was first rigorously established in \cite{Chatterjee08} under the assumption of the aforementioned \emph{superconcentration} property and the assumption that the correlations of the field are non-negative. The phenomenon of \emph{chaos} refers to an instability of the location of the maximizer with respect to small perturbations of the Gaussian field and was shown to be equivalent to superconcentration in some sense.

Our first goal in this paper is to further explore the connection between superconcentration and multiple peaks. We obtain a quantitative improvement of the number of such peaks (thus attaining an optimal bound in a certain sense) and we also remove the assumption that the correlations are non-negative.

In order to state our result properly, we need a rigorous definition of the multiple peaks property. We shall use the same definition as introduced in \cite{Chatterjee08}: Consider a sequence of \emph{centered} Gaussian fields $\X_N = \{X_{N, i}: 1\leq i\leq N\}$. Denote by  $\sigma_N^2 = \max_{1\leq i\leq N} \var X_{N, i}$ and write $[N] = \{1, \ldots, N\}$. Write $R_N(i, j) = \cov(X_{N, i}, X_{N. j})$ for all $i, j\in [N]$. In addition, define $M(\X_N) = \sup_{1\leq i\leq N} X_{N, i}$,  $m(\X_N) = \E  M(\X_N)$ and $\hat \sigma_N^2 = \var(M(\X_N))$.
\begin{definition}
A sequence of Gaussian fields $\X_N$ exhibits multiple peaks if and only if there exists $\ell_N \to \infty$, $\epsilon_N = o(\sigma_N^2)$, $\delta_N = o(m(\X_N))$ and $\gamma_N \to 0$ such that with probability at least $1-\gamma_N$, there is a set $A_N \subseteq [N]$ of cardinality at least $\ell_N$ satisfying 
\begin{enumerate}[(M.1)]
\item $|R_N(i, j)| \leq \epsilon_N$ for all $i \neq j\in A_N$. \label{eq-M-2}
\item  $X_{N, i} \geq m(\X_N) - \delta_N$ for all $i\in A_N$. \label{eq-M-3}
\end{enumerate}
\end{definition}
We have the following theorem.
\begin{theorem}\label{thm-MP-main}
Fix any positive sequences $\delta_N \leq m(\X_N)$, $\epsilon_N \leq \sigma_N^2$ and $\zeta_N\leq 1$. Then for all $N \in \mathbb{N}$, with probability at least $1 - \frac{C_1 \hat \sigma^2_N}{ \delta_N^2} - \zeta_N$ there exists $A_N \subseteq [N]$ of cardinality at least $\exp \left( \frac{C_2 \epsilon_N^2 \delta_N \zeta_N}{m(\X_N) \sigma_N^2 \hat \sigma^2_N}\right)$ such that  (M.1) and (M.2) hold. Here $C_1, C_2$ are positive universal constants. 
\end{theorem}
Quantitatively,  \cite[Theorem 3.7]{Chatterjee08} guarantees existence of such a set $A_N$ with cardinality at least $\left (\frac{\delta_N \epsilon_N}{m(\X_N) \hat \sigma_N^2} \right )^{1/3}$ as opposed to the exponential bound in Theorem~\ref{thm-MP-main}. In addition, our result does not require the non-negative correlation assumption, thereby solving Open Problem 5 in \cite{Chatterjee08}.

Another quantity that has received a significant amount of attention in the statistical physics community is the  \emph{free energy} of the field at an inverse-temperature $\beta > 0$, defined as
\begin{equation}
F_{N, \beta} = \frac{1}{\beta} \log \Big(\sum_{i=1}^N \mathrm{e}^{\beta X_{N, i}}\Big)\,.
\end{equation}
Evidently, as $\beta \to \infty$, this quantity approaches $M(\X_N)$. In view of this, it may be natural to look into the property that the quantity $F_{N, \beta}$ is concentrated around its mean for finite values of $\beta$. This phenomenon is referred to as the \emph{superconcentration of free energy} of the process at inverse temperature $\beta$. In some cases, the free-energy for certain values of $\beta$ seems to be a more tractable quantity than the supremum, and it may be easier to establish concentration bounds for the free energy than for the supremum of the field, as witnessed in \cite{Chatterjee09} regarding the S-K model (named after Sherrington and Kirkpatrick) for spin glasses (see definition below). 
The result \cite{Chatterjee09} in which Chatterjee deduced the property of multiple-peaks from superconcentration of the free energy, can be seen as an adaptation of the result in \cite{Chatterjee08}. In this paper, we also give an adaptation of Theorem \ref{thm-MP-main} to the free energy. We denote by $\hat \sigma^2_N(\beta) = \var (F_{N, \beta})$.

\begin{theorem}\label{thm-MP-free-energy-main}
Suppose that $\tilde{\sigma}_N(\beta)$ is an upper bound on $\hat \sigma_N(\beta)$ for all $N\in \N$ and $\beta \geq 0$. For any positive sequences $\delta_N \leq m(\X_N)$, $\epsilon_N \leq \sigma^2_N$, $\zeta_N\leq 1$ and $$\beta_N \geq
C_1 \max \left( \frac{ \log N}{ \delta_N }, \frac{1}{\tilde{\sigma}_N(\beta_N)},
\frac{\delta_N \epsilon_N^2 }{m(\X_N) (\tilde \sigma_N(\beta_N))^3 \sigma_N^2} \right)\,,$$ with probability at
least $1 - \frac{C_2 \sigma_N^2}{ \delta_N^2} - \zeta_N$ there exists $A_N \subset
[N]$ with cardinality at least $$\exp \left( \frac{C_3 \epsilon_N^2 \delta_N
  \zeta_N}{m(\X_N) (\tilde{\sigma}_N(\beta_N))^2 \sigma_N^2} \right)$$ such that  (M.1) and (M.2) holds. Here $C_1, C_2, C_3$ are positive universal constants.
\end{theorem}

In the preceding theorem, we work with the upper bound $\tilde{\sigma}_N(\beta)$ of $\hat \sigma_N(\beta)$ so one may verify the assumption that
$\beta_N \geq C_1 \max \left( \frac{ \log N}{ \delta_N }, \frac{1}{\tilde{\sigma}_N(\beta_N)}, \frac{\delta_N \epsilon_N^2 }{m(\X_N) (\tilde \sigma_N (\beta_N))^3 \sigma_N^2}\right)$ without knowing a lower bound on $\hat \sigma_N(\beta_N)$.

Let us now briefly discuss some applications of Theorems~\ref{thm-MP-main} and \ref{thm-MP-free-energy-main}. Our first application is for directed polymers. Let $\mathbb Z_n^2$ denote the graph whose vertices are $\{0,1,...,n\}^2$ and where two vertices are connected by an edge if they differ by $1$ in exactly one coordinate. Let $\mathcal P_n$ be the collection of all the $N=\binom{2n}{n}$ monotone paths on $\mathbb Z_n^2$ joining the left bottom corner $(0,0)$ and the right top conner $(n,n)$. Associate i.i.d.\ standard Gaussian variables $Z_e$ to each edge $e \in \mathbb Z_n^2$. The directed polymer is defined to be a Gaussian field $\{X_{N, P}: P\in \mathcal P_n\}$ where $X_{N, P} = \sum_{e\in P} Z_e$.
For this model, \cite[Theorem 8.1]{Chatterjee08} provided an upper bound of $O(n/\log n)$ on $\hat \sigma^2_N$.  Combined with Theorem~\ref{thm-MP-main}, it gives the following corollary.
\begin{cor}
There exist absolute constants $C_1, C_2>0$ such that the following statement holds for directed polymers (recall that $N = \binom{2n}{n}$).
For any $0<\delta_N\leq n$, $0<\epsilon_N\leq n$, $0<\zeta_N<1$, with probability at least $1 - \frac{C_1 n }{\delta_N^2 \log n} - \zeta_N$ there exists $A_N \subseteq [N]$ of cardinality at least $n^{C_2 \epsilon_N^2 \delta_N \zeta_N/n^3}$ satisfying  (M.1) and (M.2).
\end{cor}

We next discuss an application for the S-K model. For a hypercube $H_n = \{-1, 1\}^n$ (write $N = |H_n| = 2^n$),  the S-K model introduced in \cite{SK} can be viewed for our purposes as 
 a Gaussian field $\{X_{N, \sigma}: \sigma \in H_n\}$ with $X_{N, \sigma} = \frac{1}{\sqrt{2 n}}\sum_{i, j\in [n]} \sigma_i \sigma_j Z_{i, j}$ where $Z_{i,j}$'s are i.i.d.\ standard Gaussian variables. It is easy to see that  the variance for individual Gaussian variable is precisely $n$ and the expected supremum is  of order $n$. Indeed, the asymptotics of the free energy (and thus obtaining the expected supremum by sending $\beta \to \infty$) was established in a celebrated work \cite{Talagrand-parisi}, verifying the well-known prediction of the Parisi formula \cite{Parisi}. As for concentration, \cite[Theorem 1.5]{Chatterjee09} established an upper bound of $O(\beta n/\log n)$ on $\hat \sigma^2_N(\beta)$. Combining the variance bound and Theorem~\ref{thm-MP-free-energy-main}, we obtain the following (where we set $\beta_N $ to be of order $n/\delta_N$ and $(\tilde \sigma(\beta_N))^2$ to be  of order $n \beta_N /\log n$).
\begin{cor}
There exist absolute constants $C_1, C_2>0$ such that the following statement holds for the S-K model (recall $N = 2^n$).
For any positive $0<\delta_N\leq n$, $0<\epsilon_N\leq n$, $0<\zeta_N<1$, with probability at least $1 - \frac{C_1 n}{\delta_N^2 \log n} - \zeta_N$ there exists $A_N \subseteq [N]$ of cardinality at least $n^{C_2 \epsilon_N^2 \delta^2_N \zeta_N/n^4}$ satisfying (M.1) and (M.2).
\end{cor}
In particular, for both models we obtain that for fixed $\delta, \epsilon, \zeta>0$ with probability at least $1-\zeta$ there exists $n^{c_{\delta, \epsilon, \zeta}}$ (for a constant $c_{\delta, \epsilon, \zeta}>0$ depending only on $\delta$, $\epsilon$ and $\zeta$) sites such that the Gaussian values on these sites are within additive $\delta n$ to the expected supremum and the pairwise covariances are at most $\epsilon n$. This improves the corresponding polynomial in $\log n$ sites obtained in \cite{Chatterjee08, Chatterjee09}. While polynomially many large and near-orthogonal sites may still be far from satisfactory from the point of view of statistical physics, we remark that a stretched exponentially many large and near-orthogonal sites can be deduced from our results provided a verification of the prediction that the variances for the supremums (or the free energy at low temperatures) in both directed polymers and the S-K model are of order $n^{2/3}$ \cite{KPZ,  Palassini,  CZ11}. 

We conclude the discussion on multiple peaks by remarking that our results are optimal in the sense that one can construct Gaussian fields so that this field consisting of $N$ centered variables of variance $1$ whose supremum has variance of order $\hat \sigma^2_N$, such that the typical number near-orthogonal sites whose value is close to the supremum is of the same order as the bound in Theorem~\ref{thm-MP-main} up to the constant appearing in the exponent. Indeed, for a fixed value of $N$ and of $\sigma > 0$, define
$$
K = \left \lfloor \mathrm{e}^{1 / \sigma^2} \right \rfloor.
$$
Now, let $\X_N$ be a the Gaussian process constructed by taking $K$ independent standard Gaussian vairables, and duplicating $N/K$ identical copies of each of them to obtain $N$ variables. It is easy to check that this construction satisfies $\hat \sigma_N \sim \sigma$. Moreover, it is easily checked that for any  $\epsilon_N \leq 1/2$, and $\delta_N \leq m(\X_N)$, the set of near-orthogonal peaks (i.e., the cardinality of $A_N$ satisfying (M.1) and (M.2)) will be of order at most $e^{\frac{c \delta_N}{m(\X_N) \sigma^2}}$ with probability at least $1/2$ (for some absolute constant $c>0$), which shows that the dependence on $\hat \sigma_N$ and $\delta_N$ is tight in the sense described above. We remark, however, that there exist Gaussian fields which have significantly more large and near-orthogonal sites than what is proved in Theorem~\ref{thm-MP-main}. For instance, it was shown in \cite{CDD13} that any sequence of extremal Gaussian fields exhibit multiple peaks with exponentially many peaks (see \cite[Theorem 1.6]{CDD13} for details).

\bigskip

Next, we turn to discuss Question~\eqref{question-b}. There are a number of directions for possible improvement upon \eqref{eq:borel}. For instance, it was recently proved in \cite{CFMP11, MN12a, MN12b, Eldan13} that the unique minimizer that achieves equality in the isoperimetric inequality (from which \eqref{eq:borel} 
is deduced) is the half space and any set that genuinely differs from a half space (in some geometric sense) has a strictly larger Gaussian surface area and consequently will satisfy a stronger version of \eqref{eq:borel}. In this paper, we approach Question~\eqref{question-b} from a related but slightly different perspective, elaborated below.

One important direction of research concerned with the supremum of a Gaussian process is finding sharp estimates for the expectation of the supremum. Using the generic chaining technique and building upon the entropy bound in \cite{Dudley67}, a celebrated result (known as the \emph{majorizing measure} theorem) was developed by Fernique and Talagrand in \cite{F1, Talagrand87} which provides an estimate of the expected supremum up to a universal multiplicative constant factor. One of the  two major ingredients employed in the proof of the majorizing measure theorem is \eqref{eq:borel}. In view of this, it seems plausible that improving \eqref{eq:borel} based on information on the expected supremum may shed light toward sharpening the lost constant factor in the majorizing measure theorem, and in particular could hopefully help in determining whether a sequence of Gaussian fields is extremal in the sense that its expected supremum is nearly as large as possible with respect to $N$.

In this paper, we prove that the exponent in the large deviation bound in \eqref{eq:borel} can be improved under the assumption that the expected supremum is of the same order as that of the independent case, namely of order $\sqrt{\log N}$. While this may seem like a rather strong assumption, we would like to draw the reader's attention to the fact that it is actually satisfied by almost all of the canonical examples of Gaussian processes (in particular, the directed polymer and the S-K model). The theorem reads:

\begin{theorem}\label{thm-deviation}
Let $\{X_i \}_{i=1}^N$ be a centered Gaussian process with $\var[X_i] \leq 1$ for all $1\leq i\leq N$ and suppose that $\E \sup_{1\leq i\leq N} X_i \geq \alpha \sqrt{\log N}$ for a fixed $\alpha>0$. Then there exist an absolute constant $C>0$ and $c(\alpha)>0$ depending only on $\alpha$ such that for all $0<\beta \leq \alpha/100$ and all $N\in \mathbb N$ one has
$$\P \left ( \left |\sup_{1\leq i\leq N} X_i - \E \sup_{1\leq i\leq N} X_i \right | \geq  \beta\sqrt{\log N} \right ) \leq C N^{-\beta^2/(2-c(\alpha))}\,.$$
\end{theorem}

We remark that our current method does not provide a sharp $c(\alpha)$, and thus we did not attempt to optimize its value. The main point of Theorem~\ref{thm-deviation} is to suggest a new direction of research by demonstrating the possibility to improve \eqref{eq:borel} under the assumption of large expected supremum.  We believe that it is of significant interest to obtain a sharp estimate on $c(\alpha)$. Indeed, we ask the following open question.
\begin{question}\label{question-sharp-exponent}
Under the assumptions of Theorem~\ref{thm-deviation}, is it true that for all $\beta$ with $\E \sup_{1\leq i\leq N} X_i + \beta \sqrt{\log N} \leq \sqrt{2 \log N}$, we have 
$$\P \left (\sup_{1\leq i\leq N} X_i \geq  \E \sup_{1\leq i\leq N} X_i +  \beta\sqrt{\log N} \right ) \leq N^{-(\beta^2 +o_N(1))/(2-\alpha^2)}\,?$$
\end{question}
Note that the exponent in Question~\ref{question-sharp-exponent} is achieved by the Gaussian field $X_i = Z + Z_i$ where $Z$ and $Z_i's$ are independent Gaussian variables such that $\var Z = 1- \alpha^2/2$ and $\var Z_i = \alpha^2/2$ for all $1\leq i\leq N$.
In spirit, Theorem~\ref{thm-deviation} suggests that large expected supremum implies a good concentration property for the supremum. It turns out that the converse also holds in some sense. That is, a good concentration for the supremum implies that the expected supremum has to be large.
\begin{theorem}\label{thm-var-exp}
There exists an absolute constant $c>0$ such that for any centered Gaussian field $\{X_i: 1\leq i\leq N\}$ with $\var X_i = 1$ for all $1\leq i\leq N$, we have
$$
\left (\var \left [\sup_{1\leq i\leq N} X_i \right ] \right )^{1/2} \E \left [ \sup_{1\leq i\leq N} X_i \right  ] \geq c\,.
$$
\end{theorem}
In general, the expected supremum can be bounded from above  by (c.f., \cite[Lemma 2.1]{Chatterjee08})
\begin{equation}\label{eq-simple-upper-exp}
\E \sup_{1\leq i\leq N} X_i \leq \sqrt{2 \log N} \cdot  \max_{1\leq i\leq N} \sqrt{\var X_i}\,.
\end{equation}
Combined with Theorem~\ref{thm-var-exp}, it yields the following corollary. 
\begin{cor} Under the assumption of Theorem~\ref{thm-var-exp}, we have that for an absolute $c>0$
$$
\var \left [\sup_{1\leq i\leq N} X_i \right ] \geq \frac{c}{\log N}\,.
$$
\end{cor}

The structure of the rest of this paper is as follows: in Section \ref{SecMultiplePeaks} we prove Theorem \ref{thm-MP-main} and Theorem \ref{thm-MP-free-energy-main}. In Section \ref{SecLargeDeviations} we prove Theorem \ref{thm-deviation} and in Section \ref{SecDevExp} we prove Theorem \ref{thm-var-exp}.

\section{Superconcentration implies multiple peaks} \label{SecMultiplePeaks}

This section is devoted to the proofs of Theorems~\ref{thm-MP-main} and \ref{thm-MP-free-energy-main}. Consider a Gaussian field $\X = \{X_i: i\in S\}$. By rescaling, we can assume without loss of generality that $\var X_i \leq 1$ for all $i \in S$. To lighten notation, this normalization will be assumed throughout this section. For a set $U \subset S$, we use $\X_U$ to denote the restriction of $\X$ to the indices in $U$. Define $M(\X) = \sup_{i \in S} X_i$ and $R(i, j) = \cov(X_i, X_j)$ for $i, j\in S$. In addition, define $m = \E(M(\X))$ and $\sigma = \sigma(M(\X))$.

\subsection{Proof of Theorem~\ref{thm-MP-main}}

Theorem~\ref{thm-MP-main} can be directly deduced by applying the following result to each Gaussian field $\X_N$ in the sequence.
\begin{theorem} \label{thmpeaks}
There exist absolute positive constants $C_1, C_2$ such that the following holds. For any
  $0<\epsilon, \delta, \zeta \leq 1$ with probability at least
$ 1 - \frac{C_1 \sigma^2}{m^2 \delta^2} - \zeta$
 there exists $A \subset S$ with cardinality at least
$\exp \left( \frac{C_2 \epsilon^2 \delta \zeta}{\sigma^2} \right)$
such that $X_i \ge (1 - \delta)m(\X)$ for each $i \in A$,
  and $|R(i, j)| < \epsilon$ for each distinct $i, j \in A$.
\end{theorem}

In order to prove Theorem~\ref{thmpeaks}, we consider the random set $U_{1-\delta} = \{i\in S: X_i \geq (1-\delta) m\}$, and wish to show that we can find a large near-orthogonal subset (that is, a subset where the pair-wise correlations are at most $\epsilon$) in $U_{1-\delta}$. A preliminary and seemingly innocent question is whether we are able to find a large subset of $S$ of near-orthogonal variables.  It turns out that the concentration property for the supremum of the Gaussian field on $S$ guarantees the existence of a large near-orthogonal subset of $S$, as shown in Lemma~\ref{deterministic many orthogonal} below. In light of this it would then suffice to prove that, fixing the random set $U_{1-\delta}$ and considering an independent copy of $\X$, the supremum over $U_{1-\delta}$ exhibits a good concentration property. The main ingredients for the proof of this fact are in the content of Lemmas~ \ref{U_t bound} and \ref{decomp-conc} below. \\

We begin with the deterministic claim that any Gaussian process which exhibits superconcentration has a large subset of near-orthogonal variables.

\begin{lemma} \label{deterministic many orthogonal}
  Let $\X = \{ X_i : i \in S \}$ be a (not necessarily centered) Gaussian process such that $\var(X_i) \leq 1$ for all $i \in S$. For a given $\epsilon > 0$, if $[r, s]$ is an
  interval of length at most $\frac{\epsilon}{8}$ such that
  \begin{equation} \label{asumpinterval}
  \P(M(\X) \not\in [r, s]) < \frac{1}{4},
  \end{equation}
  then there exists $A \subset S$ such that
\begin{equation} \label{ALarge}
|A| \ge \mathrm{e}^{\frac{\epsilon^2}{32(r - s)^2}},
\end{equation}
  \noindent and for every distinct $i, j \in A$, $|R(i, j)| \le
  \epsilon$.
\end{lemma}
\begin{proof}
  Let $A \subset S$ be a maximal set (with respect to inclusion) satisfying
  $|R(i, j)| \le \epsilon$ for all $i, j \in A$. We will show that such a set must satisfy (\ref{ALarge}).

For each $i \in S$, let $b(i)$ denote the element of $A$ which maximizes $|R(i,
  b(i))|$. Note that by the maximality of $A$, we necessarily have
  $|R(i, b(i))| \ge \epsilon$ for all $i \in S$.


  We now consider the probability space underlying $\X$ as a standard
  $N$-dimensional Gaussian $\Gamma$ with density $\gamma$. Let
  $\{v_i\}_{i \in S}$ be vectors with norm at most 1 such that $X_i = \langle \Gamma,
  v_i \rangle + \mu_i$, so that $\mu_i = \E(X_i)$ and $\langle v_i,
  v_j \rangle = R(i, j)$. Define for $x \in \R^N$,

  \[ m(x) = \sup_{i \in S} \langle x, v_i \rangle + \mu_i \]
so that $m(\Gamma) \sim M(\X)$. In addition, we define
  \[ i(x) = \argmax_{i \in S} \left ( \langle x, v_i \rangle + \mu_i \right )\,.\]

  For a positive constant $c>0$ to be specified later, define a
  piecewise linear mapping $f_c: \R^N \to \R^N$ as follows: for a
  point $x \in \R^N$, let $(a, \eta)$ be the element of $A \times \{
  -1, 1 \}$ which maximizes $m(x + c \eta v_a)$, and define $f_c(x) = x
  + c \eta v_a$.

  Our next goal is to show that the function $f_c$ is injective. To do this, we fix $x \in \RR^N$, and let $(a, \eta)$ be as above. For notational convenience, write $y = f_c(x)$. Then, by the definition of $f_c$,

  \[ \langle x + c \eta v_a, v_{i(y)} \rangle + \mu_{i(y)} = \langle y, v_{i(y)} \rangle + \mu_{i(y)} = m(y) \ge \langle x \pm cv_{b(i(y))}, v_{i(y)} \rangle + \mu_{i(y)}, \]

  \noindent where the plus or minus indicates that the inequality
  holds for either choice of sign. It follows that

  \[ \left| \langle v_a, v_{i(y)} \rangle \right| \ge \left| \langle v_{b(i(y))}, v_{i(y)} \rangle \right|. \]

  \noindent On the other hand, by the definition of $b(i(y))$ we have

  \[ \left| \langle v_{b(i(y))}, v_{i(y)} \rangle \right| \ge \left| \langle v_a, v_{i(y)} \rangle \right|. \]
Now, we observe that it is legitimate to assume that the values $|\langle v_{i_1}, v_{i_2} \rangle|$ where $i_1, i_2 \in S$ and $i_1 \leq i_2$ are all distinct.
  Indeed, if this is not the case, then we may apply small random perturbation to each of the vectors $v_i$ and use the fact that the claim of the lemma is continuous with
  respect to these perturbations. Using this assumption, we may actually assume that $b(i(y)) = a$. Therefore,
  $$
  x = y - c v_{b(i(y))} \mathrm{sign}( \langle v_{b(i(y))}, y \rangle )\,,
  $$
  where $\mathrm{sgn}(x) = x/|x|$ if $x\neq 0$ and $\mathrm{sgn}(0) = 0$.
This completes the verification that $f_c$ is injective. \\

\noindent
Next, we fix $c = \frac{s - r}{\epsilon}$ and consider the set

  \[ U = \left\{ x \in \R^N :m(x) \ge r, \, \sup_{a \in A} \langle x, v_a \rangle \le \sqrt{2 \log |A|} + 3  \right\}. \]
We claim that
\begin{equation} \label{mfcu}
x \in U \Rightarrow m(f_c(x)) \geq s.
\end{equation}
Indeed, we have for all $x \in \RR^N$,
  \[ m(f_c(x)) \ge m(x \pm c v_{b(i(x))}) \ge \langle x \pm c v_{b(i(x))}, v_{i(x)} \rangle + \mu_{i(x)} \ge \langle x, v_{i(x)} \rangle + c \epsilon + \mu_{i(x)} = m(x) + c \epsilon\,. \]
In addition, under the assumption $x \in U$, we have
$$
m(f_c(x)) \geq m(x) + c \epsilon \geq r + c \epsilon = s\,,
$$
thereby proving (\ref{mfcu}). Now, equation (\ref{mfcu}) implies that
  \[ \P\left(M(\X) \ge s \right) \ge \gamma(f_c(U)) \,.\]
Therefore, we conclude from the assumption (\ref{asumpinterval}) that necessarily
\begin{equation} \label{contradictthis}
\gamma(f_c(U)) \leq \frac{1}{4}.
\end{equation}

In the following, we will suppose for the sake of contradiction that equation (\ref{ALarge}) is not satisfied and conclude that $\gamma(f_c(U)) > \frac{1}{4}$, thus concluding the lemma. \\

By \eqref{eq-simple-upper-exp}, we have  $\E(\sup_{i \in A} \langle \Gamma, v_i \rangle ) \le \sqrt{2 \log |A|}$. Furthermore, a simple application
of \eqref{eq:borel} show that the above maximum is relatively concentrated around its expectation in the sense that
  \[ \P \left (\sup_{i \in A} \langle \Gamma, v_i \rangle \ge \sqrt{2 \log |A|} + 3 \right ) \le \frac{1}{8}.\]

  \noindent We also have by hypothesis

  \[ \P(M(\X) \le r) \le \frac{1}{4}. \]

  \noindent Thus, $\gamma(U) \ge \frac{5}{8}$. Note that for any $x, v
  \in \R^N$ with $|v| \leq  1$ and $\langle x, v \rangle \le R$, we have

  \[ \mathrm{e}^{-\frac{\|x + cv\|^2}{2}} \ge \mathrm{e}^{-cR - \frac{c^2}{2}}\mathrm{e}^{-\frac{\|x\|^2}{2}}. \] 

  \noindent Thus,

  \[ \gamma(f_c(U)) \ge \mathrm{e}^{-c(\sqrt{2 \log |A|} + 3) - \frac{c^2}{2}} \gamma(U). \]

Recall the hypothesis that
  $s - r \le \frac{\epsilon}{8}$, so $c \le \frac{1}{8}$. Also, note
  that $c \sqrt{2 \log |A|} \le \frac{1}{4}$. Hence,

  \[ \gamma(f_c(U)) \ge \mathrm{e}^{-c\sqrt{2 \log |A|} - \frac{1}{2}} \gamma(U)  \ge \mathrm{e}^{-\frac{3}{4}} \cdot \frac{5}{8} > \frac{1}{4} \]
which contradicts (\ref{contradictthis}), and the lemma is proven.
\end{proof}

For a Gaussian process $\X = \{X_i, i \in S \}$ and a set $U \subset S$, we recall that $\X_{U} = \{X_i, i \in U\}$ is the
process $\X$ restricted to the set $U$.

\begin{lemma} \label{U_t bound}
  Let $\X$ be a centered Gaussian process, with $m = \E(M(\X))$ and
  $\sigma = \sigma(M(\X))$. For any real number $t \in (0, 1)$, define
  $U_t = \{ i \mid X_i \ge tm \}$. Then for any $\lambda > 0$,

  \[ \P \left( M(\X'_{U_t}) \ge \sqrt{1 - t^2} \cdot m + \frac{\lambda}{\sqrt{1 - t^2}} \right) \le \frac{\sigma^2}{\lambda^2}, \]

  \noindent where $M(\X'_{U_t}) = \sup_{i\in U_t} X'_i$ and $\X'$ is an independent copy of $\X$.
\end{lemma}
\begin{proof}
  Let $\X'' = t \X + \sqrt{1 - t^2}\X'$, so that $\X''$ is equal in
  distribution to $\X$. Then,

  \[ M(\X'') \ge \sup_{i \in U_t} \X''_i \ge t^2m +\sqrt{1 - t^2} M(\X'_{U_t}) \]
  \[ M(\X'_{U_t}) \le \frac{M(\X'') - t^2m}{\sqrt{1 - t^2}}. \]

  \noindent By Chebyshev's inequality,

  \[ \P(M(\X'') \ge m + \lambda) \le \frac{\sigma^2}{\lambda^2}, \]

  \noindent which gives the desired inequality.
\end{proof}

\begin{lemma} \label{decomp-conc}
  Let $\X$ be a centered Gaussian process, with $m = \E(M(\X))$ and
  $\sigma = \sigma(M(\X))$. For any $t \in (\frac{1}{2}, 1)$, we may
  write $\X = t\X' + \sqrt{1 - t^2}\X''$, where $\X'$ and $\X''$ are
  independent copies of $\X$. For $s \in (0, 1)$, define $U'_s = \{ i
  \in S : \X'_i \ge sm \}$. If $\lambda > 1 - t$, then

  \[ \P \left( \sup_{i \in S, \, i \not\in U'_{t - \lambda}} \X_i \ge \left( 1 - c_1\lambda \right) m \right) \le \frac{C_2 \sigma^2}{m^2 \lambda^2}, \]
  where $c_1$ and $C_2$ are absolute constants.
\end{lemma}
\begin{proof}
  Let $s$ be a non-negative real number such that $s \le t -
  \lambda$. Note that the condition $\lambda > 1 - t$ ensures $t - s > 1
  - t$. We have the elementary inequality

  \begin{align} \label{steq}
    1 - st - \sqrt{(1 - s^2)(1 - t^2)} &\ge \frac{1}{2} \left( \sqrt{1 - s^2} - \sqrt{1 - t^2} \right)^2 \nonumber\\
    &= \frac{1}{2} \cdot \frac{(t - s)^2(t + s)^2}{\left( \sqrt{1 - s^2} + \sqrt{1 - t^2} \right)^2}
    \ge \frac{(t - s)^2}{32(1 - t^2)} \ge \frac{(t - s)^2}{64(1 - t)} \ge \frac{t - s}{C}\,,
  \end{align}
  where $C = 64$ is an absolute constant.

  Let $\Delta = \frac{\lambda}{4C}$, and consider the random set $V'_s
  = \{ i \in S : \X'_i \in [sm, (s + \Delta)m] \}$. Note that $V'_s
  \subset U'_s$. Suppose that $M(\X_{V'_s}) \ge \left( 1 - \frac{t -
    s}{4C} \right) m$. Then,

  \[ tM(\X'_{V'_s}) + \sqrt{1 - t^2}M(\X''_{V'_s}) \ge M(\X_{V'_s}) \ge \left( 1 - \frac{t - s}{4C} \right) m \]
  Since $M(\X'_{V'_s}) \le (s + \Delta)m$ by the definition of $V'_s$,
  rearranging the above inequality yields

  \[ \sqrt{1 - t^2}M(\X''_{V'_s}) \ge \left( 1 - \frac{t - s}{4C} - t(s + \Delta) \right) m \ge \left( 1 - ts - \frac{t - s}{2C} \right) m \]
  \[ \ge \left( \sqrt{(1 - s^2)(1 - t^2)} + \frac{t - s}{2C} \right) m, \]
  where in the second inequality we used $t\Delta \le \Delta \le
  \frac{t - s}{4C}$, and in the last inequality we used
  (\ref{steq}). Dividing through, we have

  \[ M(\X''_{V'_s}) \ge m \sqrt{1 - s^2} + \frac{(t - s) m}{2C \sqrt{1 - t^2}} \ge m \sqrt{1 - s^2} + \frac{(t - s)m}{2C \sqrt{1 - s^2}}. \]
  Since $M(\X''_{U'_s}) \ge M(\X''_{V'_s})$, we have

  \[ \P \left( M(\X_{V'_s}) \ge \left( 1 - \frac{t - s}{4C} \right) m \right) \le \P \left( M(\X''_{U'_s}) \ge m \sqrt{1 - s^2} + \frac{(t - s) m}{2C \sqrt{1 - s^2}} \right) \le \frac{4C^2\sigma^2}{m^2(t - s)^2}, \]
  where the last inequality follows from Lemma \ref{U_t bound}.

  The above inequality essentially says that the indices $i$ for which
  $X'_i \in [sm, (s + \Delta)m]$ cannot be large in $\X$. We now sum
  over $s = 0, 1, \ldots , \left\lfloor \frac{(t - \lambda)}{\Delta}
  \right\rfloor$.

\begin{align*}\P \left( \sup_{i \in S, \, i \not\in U'_{t - \lambda}} \X_i \ge \left( 1 - \frac{\lambda}{4C} \right) m \right) &\le \sum_{k = 0}^{\left\lfloor \frac{(t - \lambda)}{\Delta} \right\rfloor} \P \left(M(\X_{V'_{\Delta k}}) \ge \left( 1 - \frac{\lambda}{4C} \right) m \right) \\
&\le \sum_{k = 0}^{\left\lfloor \frac{(t - \lambda)}{\Delta} \right\rfloor} \P \left(M(\X_{V'_{\Delta k}}) \ge \left( 1 - \frac{t - \Delta k}{4C} \right) m \right) \\
  & \le \sum_{k = 0}^{\left\lfloor \frac{(t - \lambda)}{\Delta} \right\rfloor} \frac{4C^2\sigma^2}{m^2(t - \Delta k)^2} \le \frac{4C^2\sigma^2}{m^2\Delta^2} \sum_{\ell = 1}^\infty \frac{1}{\ell^2}. \end{align*}

  \noindent Since the last series converges, and $\Delta =
  \frac{\lambda}{4C}$, this proves the lemma.
\end{proof}

\begin{cor} \label{decomp-conc-cor}
  Let $\X = \{X_i; ~ i \in S\}$ be a centered Gaussian process such that $\var[X_i] \leq 1$ for all $i \in S$. For a given $\delta > 0$, let $\alpha = 1 - \frac{\delta}{4}$. Write

  \[ \X = \alpha \X' + \sqrt{1 - \alpha^2} \X'', \]

  \noindent where $\X'$ and $\X''$ are independent copies of
  $\X$. Then,

  \[ \P \left (\X'_{i(\X)} \ge (1 - \delta) m \right ) \ge 1 - \frac{C \sigma^2}{m^2 \delta^2} \]

  \noindent for an absolute constant $C>0$, where $\sigma^2 = \var[M(X)]$.
\end{cor}
\begin{proof}
  We use the notation of Lemma \ref{decomp-conc}, with $t = \alpha$
  and $\lambda = \frac{3}{4} \delta$. Then,
\begin{align*} \P \left (\X'_{i(\X)} < (1 - \delta) m \right ) &= \P\left(i(\X) \not\in U'_{1 - \delta}\right) \\
  & \le \P \Big( M(\X) \le (1 - \frac{3}{4}c_1 \delta)m \Big) + \P \left( \sup_{i \in S, \, i \not\in U'_{1 - \delta}} \X_i \ge (1 - \frac{3}{4}c_1\delta) m \right) \\
 &\le \frac{16\sigma^2}{9c_1^2m^2\delta^2} + \frac{16C_2 \sigma^2}{9m^2 \delta^2},
  \end{align*}
  by Chebyshev's inequality and Lemma \ref{decomp-conc}.
\end{proof}

We are finally ready to prove the connection between superconcetration and multiple peaks.

\begin{proof}[Proof of theorem \ref{thmpeaks}]
We begin fixing some $\zeta > 0$, denoting $\alpha = 1 -
\frac{\delta}{4}$ and considering the decomposition $\X = \alpha \X' +
\sqrt{1 - \alpha^2} \X''$, where $\X'$ and $\X''$ are independent
copies of $\X$. Define
$$
U' = \bigl \{ i \mid X'_i \ge (1 - \delta) \E(M(\X')) \bigr \}.
$$
Since $\X'$ has the same distribution as $\X$, it is enough to show that with probability at least
$1 - 4C\sigma^2/(m^2 \delta^2) - \zeta$ there exists a subset $A \subset U'$ such that
\begin{equation} \label{Abig}
|A| \geq \mathrm{e}^{\frac{\zeta \epsilon^2 (1-\alpha^2)}{32 \sigma^2}} \geq \mathrm{e}^{\frac{\zeta \epsilon^2 \delta}{128 \sigma^2}}
\end{equation}
and
\begin{equation} \label{Aortho}
i,j \in A, i \neq j \Rightarrow |R(i,j)| \leq \epsilon
\end{equation}
where $C>0$ is a universal constant. \
To this end, we consider
$$
\mathbf{Y} (\X') = \frac{\alpha \X'_{U'}}{\sqrt{1 - \alpha^2}} + \X''_{U'}\,.
$$
It is convenient to separate the two sources of randomness in $\mathbf Y(\X')$. In what follows, we will condition on the realization of $\X'$ (therefore also $U'$), and consider
$\mathbf{Y}(\X')$ as a non-centered Gaussian process indexed over set $U'$, where the randomness comes from the process $\X''$ and the mean vector is given by $\frac{\alpha \X'_{U'}}{\sqrt{1 - \alpha^2}}$.  Define
$$
g(\X') =    \left . \P \left( M(\mathbf{Y}(\X')) \in \left[ \frac{m - \sigma \zeta^{-1/2} }{\sqrt{1 - \alpha^2}}, \frac{m + \sigma \zeta^{-1/2} }{\sqrt{1 - \alpha^2}} \right] \right \vert \X' \right)
$$
and let $E$ be the event that $\{g(\X') \geq 3/4\}$. Note that $E$ is measurable in the $\sigma$-field generated by $\X'$. Clearly, whenever the event $E$ holds,
the Gaussian process $(\mathbf{Y}(\X'))$ (conditioned on $\X'$ in the aforementioned manner) satisfies the assumption (\ref{asumpinterval}) with
$$
s - r = \frac{2 \sigma}{\sqrt{\zeta} \sqrt{1 - \alpha^2}}.
$$
Therefore, by applying lemma \ref{deterministic many orthogonal} we learn that whenever $E$ holds there exists a subset $A \subset U'$ satisfying (\ref{Abig}) and (\ref{Aortho}). It thus remains to show that
\begin{equation} \label{pFeq}
\PP(E) > 1 - 4C\sigma^2/(m^2 \delta^2) - \zeta.
\end{equation}
To this end, we define another event
$$
F = \left \{ X'_{i(\X)} \ge (1 - \delta)m \right \}.
$$
By Corollary \ref{decomp-conc-cor}, we have
\begin{equation} \label{eqPe}
\P(F) > 1 - \frac{C\sigma^2}{m^2\delta^2}\,,
\end{equation}
where $C>0$ is a universal constant. Now, we remark that whenever the event $F$ holds, one has $$\sqrt{1 - \alpha^2} M( \mathbf{Y}(X')) = M(\mathbf{X}).$$
It follows from the definition of $g(\cdot)$ that
$$
\E g(\X') \geq  \P \left  (M(\X) \in \left [m- \frac{\sigma}{\sqrt{\zeta}}, m+ \frac{\sigma}{\sqrt{\zeta}} \right ] \right ) - \P \left (F^C \right )  \geq
1 - C\sigma^2/(m^2 \delta^2) - \zeta.
$$
where in the second passage, we used (\ref{eqPe}) and Chebychev's inequality.
An application of Markov's inequality with the last equation establishes (\ref{pFeq}), and the proof is complete.
\end{proof}

\subsection{Proof of Theorem~\ref{thm-MP-free-energy-main}}
For $\beta > 0$, define $F_\beta = (\frac{1}{\beta} \log \sum_{i\in S} \mathrm{e}^{\beta X_i})$, and denote by $\sigma_\beta^2 = \var F_\beta$.
Theorem~\ref{thm-MP-free-energy-main} can be deduced by applying the following result to each Gaussian field in the sequence.

\begin{theorem} \label{thmpeaks-free-energy}
There exist constants $C_1, C_2, C_3 >0$ such that the
following holds. Suppose that $ \sigma_\beta \leq \tilde \sigma_\beta$ for all $\beta \geq 0$. For any $0 < \delta, \epsilon, \zeta < 1$ and all $\beta \geq
C_1 \max \left( \frac{\log N}{ \delta m}, \frac{1}{\tilde{\sigma}_\beta},
\frac{\delta \epsilon^2}{\tilde{\sigma}_\beta^3} \right)$, with probability at
least $1 - \frac{C_2}{m^2 \delta^2} - \zeta$ there exists $A \subset
S$ with cardinality at least $\exp \left( \frac{C_3 \epsilon^2 \delta
  \zeta}{\tilde{\sigma}_\beta^2} \right)$ such that $X_i \ge (1 -
\delta)m(\X)$ for each $i \in A$, and $|R(i, j)| < \epsilon$ for each
distinct $i, j \in A$.
\end{theorem}

The proof is similar to  that of
Theorem~\ref{thm-MP-main}, but a few changes are needed. In what
follows, we will omit details which are repeated from the proof of
Theorem~\ref{thm-MP-main} and highlight the differences. The reader is
advised to become familiar with the previous subsection before reading
this one.

First, we need a version of Lemma \ref{deterministic many orthogonal}
which deals with free energies.

\begin{lemma} \label{deterministic many orthogonal free}
  Let $\X = \{ X_i : i \in S \}$ be a (not necessarily centered)
  Gaussian process such that $\var(X_i) \leq 1$ for all $i \in S$. For
  a given $\epsilon > 0$, suppose that $[r, s]$ is an interval of
  length at most $\frac{\epsilon}{8}$ such that
  \begin{equation} \label{asumpinterval-beta}
  \P(F_\beta(\X) \not\in [r, s]) < \frac{1}{4},
  \end{equation}
  Furthermore, suppose that
  \begin{equation}
    \beta \ge \frac{\epsilon^2}{128(s - r)^3}.
  \end{equation}
  Then, there exists $A \subset S$ such that
\begin{equation} \label{ALargeBeta}
|A| \ge \mathrm{e}^{\frac{\epsilon^2}{128(r - s)^2}},
\end{equation}
  \noindent and for every distinct $i, j \in A$, $|R(i, j)| \le
  \epsilon$.
\end{lemma}
\begin{proof}
  Define $A \subset S$ and $b: S \to A$ as in the proof of Lemma
  \ref{deterministic many orthogonal}. Our goal is to show that $A$
  satisfies (\ref{ALargeBeta}). Suppose for the sake of contradiction that
  it does not.

  As before, consider the probability space underlying $\X$ as a
  standard $N$-dimensional Gaussian $\Gamma$ with density
  $\gamma$. Let $\{v_i\}_{i \in S}$ be vectors with norm at most 1
  such that $X_i = \langle \Gamma, v_i \rangle + \mu_i$, so that
  $\mu_i = \E(X_i)$ and $\langle v_i, v_j \rangle = R(i, j)$. Define
  for $x \in \R^N$,

  \[ m_\beta(x) = \frac{1}{\beta} \log \sum_{i \in S} \mathrm{e}^{\beta (\langle x, v_i \rangle + \mu_i)} \]
  so that $m_\beta(\Gamma) \sim F_\beta(\X)$. Let us also define for
  $a \in A$ and $\chi \in \{-1, 1\}$ the quantities

  \[ g_\beta^\chi (x, a) = \sum_{\substack{j\in S, \, b(j)=a \\ \mathrm{sgn}(v_a, v_j) = \chi}} \mathrm{e}^{\beta (\langle x, v_j\rangle + \mu_j)} \,. \]
  \[ \hat{m}_\beta(x) = \frac{1}{\beta} \sup_{a \in A, \chi\in \{-1, 1\}} g_\beta^\chi(x, a) \,. \]
  Evidently, we have

  \[ m_\beta(x) \ge \hat{m}_\beta(x) \ge m_\beta(x) - \beta^{-1}\log (2|A|). \]

  For a positive constant $c>0$ to be specified later, define a
  piecewise linear mapping $f_c: \R^N \to \R^N$ as follows: for a
  point $x \in \R^N$, let $(a, \eta)$ be the element of $A \times \{
  -1, 1 \}$ which maximizes $\hat{m}_\beta(x + c \eta v_a)$, and
  define $f_c(x) = x + c \eta v_a$. We next verify that
  $f_c$ is injective outside of a set of probability zero.

  Write $y = f_c(x)$, and let

  \[ (\hat i(y), \chi(y)) = \argmax_{a \in A, \chi\in \{-1, 1\}} g_\beta^\chi(y, a). \]
  By definition of $(a, \eta)$, we get that
  \[ g_\beta^{\chi(y)} (x + c\eta v_a, \hat i(y))  = g_\beta^{\chi(y)} (y, \hat i(y)) \geq g_\beta^{\chi(y)} (x \pm c v_{\hat i(y)}, \hat i(y)) \,. \]
  On the other hand, by definition of $b(i(y))$, we see that
  $$\max(g_\beta^{\chi(y)} (x + c v_{\hat i(y)}, \hat  i(y)) , g_\beta^{\chi(y)} (x - c v_{\hat i(y)}, \hat i(y)) )\geq g_\beta^{\chi(y)} (x + c \eta v_a, \hat  i(y))\,.$$
Altogether, we deduce that
\begin{equation}\label{eq-reconstruct-a-eta}
\max(g_\beta^{\chi(y)} (x + c v_{\hat i(y)}, \hat i(y)) , g_\beta^{\chi(y)} (x - c v_{\hat i(y)}, \hat  i(y)) )=g_\beta^{\chi(y)} (x + c \eta v_a, \hat i(y))\,.
\end{equation}
As in the proof of Lemma~\ref{deterministic many orthogonal},   it is legitimate to assume that the values $|\langle v_{i_1}, v_{i_2} \rangle|$ where $i_1, i_2 \in S$ and $i_1 \leq i_2$ are all distinct. Under this assumption, we see that almost surely with respect to $x\sim \gamma$ we have that the values $\{g_\beta^{\pm 1} (x + c \chi v_i, \hat i(x+ c\chi v_i)): \chi \in \{-1, 1\}, i\in A\}$ are all distinct. Combined with \eqref{eq-reconstruct-a-eta} (note that the left hand side of \eqref{eq-reconstruct-a-eta} is a function of $y$), it follows that for almost surely every given $y$ we could reconstruct $(a, \eta)$ (and thus $x$), thereby completing the verification of the injectivity of $f_c$.

  We now take $c = \frac{\beta^{-1}\log (2|A|) + s - r}{\epsilon}$ and
  consider the set

  \[ U = \left\{ x \in \R^N :m_\beta(x) \ge r, \, \sup_{a \in A} \langle x, v_a \rangle \le \sqrt{2 \log |A|} + 3  \right\}. \]
  We have for all $x \in \RR^N$,
\begin{align*}m_\beta(f_c(x)) &\ge \hat{m}_\beta(f_c(x)) \ge \hat{m}_\beta(x) + c \epsilon \\
  &= \hat{m}_\beta(x) + \beta^{-1} \log (2|A|) + (s - r) \ge m_\beta(x) + (s - r).
  \end{align*}
  Thus, we have
  \begin{equation} \label{mfcu}
    x \in U \Rightarrow m(f_c(x)) \geq s.
  \end{equation}

  The rest of the proof proceeds in exactly the same manner as the
  proof of Lemma \ref{deterministic many orthogonal}. The only
  difference is that we have chosen a different value of $c$. However,
  by the hypothesis that (\ref{ALargeBeta}) is not satisfied, we have

  \[ \log |A| < \frac{\epsilon^2}{128(s - r)^2}. \]
  The lower bound condition on $\beta$ then implies

  \begin{equation} \label{CBound}
  c = \frac{\beta^{-1} \log |A| + s - r}{\epsilon} \le \frac{2(s - r)}{\epsilon}
  \end{equation}
  We thus have

  \[ c \sqrt{2 \log |A|} \le \frac{1}{4}. \]
  Also, combining (\ref{CBound}) with the condition that $s - r \le
  \frac{\epsilon}{16}$, we obtain

  \[ c \le \frac{1}{8}. \]
  These are the only properties of $c$ needed in the proof of Lemma
  \ref{deterministic many orthogonal}, so the same argument works.
\end{proof}

To prove Theorem \ref{thmpeaks-free-energy}, we use the same setup as
the proof of Theorem \ref{thmpeaks}.  Let $\alpha = 1 -
\frac{\delta}{4}$, and make the decomposition $\X = \alpha \X' +
\sqrt{1 - \alpha^2} \X''$. Let $U' = \{ i \in S : \X'_i \ge (1 -
\delta)m \}$.  Since we are not assuming any superconcentration of
$M(\X)$, we will implicitly use the bound $\sigma^2(M(\X)) \le 1$.

Recall that in the proof of Theorem \ref{thmpeaks}, we use Corollary
\ref{decomp-conc-cor} to show that the maximum comes from the indices
in $U'$ with high probability. The next lemma is a similar statement
for the free energy; although all indices contribute to the free
energy, we show that with high probability, most of the contribution
comes from indices in $U'$.

\begin{lemma} \label{free-energy-contrib}
  There exists a universal constant $C$ such that whenever $\beta > \frac{C \log
    N}{\delta m}$, we have

  \[ F_\beta(\X) \ge F_\beta(\X_{U'}) \ge F_\beta(\X) - \frac{1}{\beta}. \]
  with probability at least $1 - \frac{C}{\delta^2m^2}$.

\end{lemma}
\begin{proof}
  Clearly, $F_\beta(\X) \ge F_\beta(\X_{U'})$ holds
  deterministically. Thus, we focus our attention on the second
  inequality. According to Lemma \ref{decomp-conc}, we have $M(\X_{S
    \setminus U'}) \le (1 - c_1 \delta)m$ with probability at least $1
  - \frac{C_2}{\delta^2m^2}$. Furthermore, Chebyshev's inequality
  tells us that $M(\X) \ge (1 - c_1 \delta / 2)m$ with probability at
  least $1 - \frac{1}{c_1^2\delta^2m^2}$.

  Let $\delta' = c_1 \delta / 2$. Then, excluding events of
  probability at most $\frac{C}{\delta^2m^2}$, where $C$ is a universal
  constant, we may assume that $M(\X) \ge (1 - \delta')m$ and $M(\X_{S
    \setminus U'}) \le (1 - 2\delta')m$.
    In that case,
  \begin{align*} F_\beta(\X) &= \frac{1}{\beta} \log \sum_{i \in S} e^{\beta X_i} = \frac{1}{\beta} \log \left( \sum_{i \in U'} e^{\beta X_i} + \sum_{i \in S \setminus U'} e^{\beta X_i} \right) \\
  & \le \frac{1}{\beta} \log \left( \sum_{i \in U'} e^{\beta X_i} + N e^{\beta (1 - 2 \delta') m} \right) \le \frac{1}{\beta} \log \left( \sum_{i \in U'} e^{\beta X_i} + N e^{-\beta \delta' m} e^{\beta (1 - \delta') m} \right) \\
  &\le \frac{1}{\beta} \log \left( \sum_{i \in U'} e^{\beta X_i} + N e^{-\beta \delta' m} \cdot e^{\beta M(\X)} \right) \le \frac{1}{\beta} \log \left( \sum_{i \in U'} e^{\beta X_i} + N e^{-\beta \delta' m} \sum_{i \in U'} e^{\beta X_i} \right) \\
  &= F_\beta(\X_{U'}) + \frac{1}{\beta} \log (1 + Ne^{-\beta \delta' m}).
    \end{align*}
  If $C$ is taken to be a sufficiently large universal constant, the assumption $\beta > \frac{C \log N}{\delta m}$
  implies the second term in the last expression is bounded by
  $\frac{1}{\beta}$. This proves the lemma.
\end{proof}

\begin{proof}[Proof of Theorem \ref{thmpeaks-free-energy}]
  Recall the strategy of proving Theorem \ref{thmpeaks}: we first show
  that concentration of $F_\beta(\X)$ implies with high probability a
  concentration of $F_\beta(\X''_{U'} + \mu)$ for some $\mu$, and then
  we apply Lemma \ref{deterministic many orthogonal free} to show the
  existence of many near-orthogonal indices in $U'$, which proves
  the theorem.

  Let $f_\beta = \E(F_\beta(\X))$. By Chebyshev's inequality,

  \begin{equation} \label{chebyshev-beta}
    \P \left(F_\beta(\X) \in [f_\beta - \tilde{\sigma}_\beta \zeta^{-1/2}, f_\beta - \tilde{\sigma}_\beta \zeta^{-1/2}] \right) \ge 1 - \frac{\sigma_\beta^2 \zeta}{\tilde{\sigma}_\beta^2} \ge 1 - \zeta.
  \end{equation}
  Define the process

  \[ \mathbf{Y} (\X') = \frac{1}{\sqrt{1 - \alpha^2}} \cdot \X_{U'} = \frac{\alpha \X'_{U'}}{\sqrt{1 - \alpha^2}} + \X''_{U'}, \]
  (where, as above, $\alpha = 1 - \delta / 4$) and let $\beta' = \sqrt{1 - \alpha^2} \beta$. Recall the hypotheses
  that $\beta \ge \frac{C_1 \log N}{\delta m}$ and $\beta \ge
  \frac{C_1}{\tilde{\sigma}_\beta}$, and take $C_1 \ge \max(1, C)$, where $C$
  is the constant of Lemma \ref{free-energy-contrib}. Then, Lemma
  \ref{free-energy-contrib} tells us that

  \[ \bigl |\sqrt{1 - \alpha^2} \cdot F_{\beta'}(\mathbf{Y}) - F_\beta(\X) \bigr | = |F_\beta(\X_{U'}) - F_\beta(\X)| \le \frac{1}{\beta} \le \tilde{\sigma}_\beta \zeta^{-1/2} \]
  with probability at least $1 - \frac{C}{\delta^2m^2}$. Thus,
  combining with (\ref{chebyshev-beta}), $F_{\beta'}(\mathbf{Y})$ lies
  in an interval $[r, s]$ of size

  \[ \frac{4 \tilde{\sigma}_\beta}{\sqrt{\zeta} \sqrt{1 - \alpha^2}} \]
  with probability at least $1 - \zeta - \frac{C}{\delta^2m^2}$. As in
  the proof of Theorem \ref{thmpeaks}, define

  \[ g(\X') = \P(F_{\beta'}(\mathbf{Y}(\X')) \in [r, s] ~ \mid \X'), \]
so according to the above we have
\begin{equation} \label{egx}
\EE[g(\X')] \geq 1 - \zeta - \frac{C}{\delta^2m^2}.
\end{equation}
Recall that $\beta' = \sqrt{1 - \alpha^2} \beta$.  By our assumption on $\beta$, we get that
\begin{align*}
   \beta'  \ge \frac{C_1 \sqrt{1 - \alpha^2} \cdot \delta \epsilon^2}{\tilde{\sigma}_\beta^3} \ge \frac{2C_1 (\sqrt{1 - \alpha^2})^3 \epsilon^2}{\tilde{\sigma}_\beta^3}  \ge \frac{128C_1 (\sqrt{1 - \alpha^2})^3 \zeta^{\frac{3}{2}} \epsilon^2}{64\tilde{\sigma}_\beta^3} = \frac{128C_1 \epsilon^2}{(s - r)^3}.
  \end{align*}
  Thus, on the event $g(\X') > \frac{3}{4}$ and taking $C_1$
  sufficiently large, the hypotheses of Lemma \ref{deterministic many
    orthogonal free} are fulfilled with $s - r = \frac{4 \tilde{\sigma}_\beta}{\sqrt{\zeta} \sqrt{1 - \alpha^2}}$ and inverse temperature
  $\beta'$. It follows that $U'$ contains at least

  \[ \exp\left( \frac{\epsilon^2}{128 (s - r)^2} \right) = \exp\left( \frac{\epsilon^2 \zeta (1 - \alpha^2)}{2048\tilde{\sigma}_\beta^2}\right) \ge \exp\left( \frac{C_3 \epsilon^2 \zeta \delta}{\tilde{\sigma}_\beta^2}\right) \]
  indices whose pairwise covariances do not exceed $\epsilon$ in
  magnitude, as required. Finally equation \eqref{egx}, combined with Markov's inequality,
  teaches us that $g(\X') >
    \frac{3}{4}$ occurs with probability at least $1 - 4\zeta - \frac{4C}{\delta^2m^2}$, which completes the proof.
\end{proof}

\section{A large deviation bound based on expectation} \label{SecLargeDeviations}
This section is devoted to the proof of Theorem \ref{thm-deviation}. The proof is based on stochastic calculus, and we need some preliminary notation. For a continuous martingale $M_t$ adapted to a filtration $\mathcal{F}_t$, we denote by $[M]_t$ the quadratic variation of $M_t$ between time $0$ and $t$. By $d M_t$ we denote the It\^{o} differential of $M_t$, which we understand as a predictable process $\sigma_t$ such that $M_t$ satisfies the stochastic differential equation $d M_t = \sigma_t d W_t$ where $W_t$ is a standard Wiener process.

Fix a Gaussian field $\X = \{X_i, 1 \leq i \leq N \}$ such that $\var[X_i] \leq 1$ for all $1 \leq i \leq N$. Now, take $(B_t)_{t \geq 0}$ to be a be a standard Brownian motion in $\RR^N$ with a corresponding filtration $\mathcal{F}_t$. Clearly, there exist vectors $\{v_i: 1\leq i\leq N\}$ of Euclidean norms at most $1$ such that we can represent the Gaussian field $\X$ by
$$X_i = \langle v_i, B_1\rangle \mbox{ for every } 1\leq i\leq N\,.$$
Define $f: \RR^N \mapsto \RR$ by $f(x) = \sup_{1\leq i\leq N} \langle v_i, x \rangle$ so that
$$
f(B_1) \sim \sup_{1 \leq i \leq N} X_i.
$$
Our goal is to show that $f(B_1)$ admits a large-deviation bound. A central component of the proof will be the Doob martingale
$$
S_t = \EE[ f(B_1) | \mathcal{F}_t ]\,,
$$
generated by the random variable $f(B_1)$ and filtration $\mathcal F_t$. Thanks to the Dambis/Dubins-Schwartz theorem, we can then view $(S_t)_{0\leq t\leq 1}$ as a time change of (one-dimensional) Brownian motion stopped at some random time $\tau=[S]_1$ (which corresponds to $t=1$). The main idea is that, due to the Gaussian concentration of the maximum for a Brownian motion stopped before time $T$, it will suffice to prove that with overwhelming probability $\tau$ is strictly less than $1$. To this end, we will try to calculate $d [S]_t$ by means of It\^{o} calculus, in what follows.

For $v\in\RR^N$ and $\sigma>0$, define
$$
\gamma_{v, \sigma} (x) = \frac{1}{\sigma^N (2 \pi)^{N/2}} \exp \left (- \frac{1}{2 \sigma^2} |x - v|^2 \right ) \mbox{ for } x\in \RR^N\,.
$$
An elementary property of the Brownian motion is that the distribution of $B_1$ conditioned on $\mathcal{F}_t$ has density $\gamma_{B_t, \sqrt{1-t}}(x)$. Therefore, we have that
$$
S_t = \int_{\RR^N} f(x) \gamma_{B_t, \sqrt{1-t}} (x) dx.
$$
For convenience of notation, we write
$$
F_t(x) = \gamma_{B_t, \sqrt{1-t}}(x).
$$
A direct calculation carried out in  \cite[Lemma 7]{Eldan13} gives that
$$
d F_t(x) = (1-t)^{-1} F_t(x) \langle x - B_t, d B_t \rangle\,.
$$
As a result of the above equation, we can calculate
$$
d S_t = d \int_{\RR^N} f(x) F_t(x) dx = (1-t)^{-1} \left \langle \int_{\RR^N} f(x) (x - B_t) F_t(x), d B_t \right  \rangle\,,
$$
and therefore we obtain
$$
d [S]_t = (1-t)^{-2} \left | \int_{\RR^N} (x - B_t) f(x) F_t(x) dx  \right |^2 dt\,,
$$
where we recall that $[S]_t$ denotes the quadratic variation for process $(S_t)$.
Substituting $y = \frac{x - B_t}{\sqrt{1-t}}$ in the last equation, we get that
\begin{equation}\label{eq-d-[S]}
d [S]_t = (1-t)^{-1} \left | \int_{\RR^N} y f(\sqrt{1-t} y + B_t) d \gamma(y)  \right |^2 dt.
\end{equation}
For convenience, we denote
$$
g_t(x) = \frac{ f(\sqrt{1-t} x + B_t) - f(B_t) }{\sqrt{1-t} }\,.
$$
Plugging this definition into \eqref{eq-d-[S]}, and using the fact that $\int_{\RR^N}x d \gamma(x) = 0$ gives
$$
d [S]_t = V_t dt,
$$
where $V_t$ is defined as
\begin{equation}\label{eq-g-t-V-t}
V_t = \left | \int_{\RR^N} x g_t(x) d \gamma(x) \right |^2\,.
\end{equation}
We wish to show that $V_t$ is strictly less than 1 for a strictly positive time interval. To this end, let
 $\epsilon, \delta > 0$ be two small numbers to be fixed, and define two events
\begin{equation}\label{eq-def-E-12}
E_1 = \left  \{V_t \leq 1 - \epsilon, ~~ \forall 0 \leq t \leq \delta \right  \}\,, \mbox{ and }
E_2 = \left \{  f(B_t) \leq \frac{\alpha}{2} \sqrt{ \log N}, ~~ \forall 0 \leq t \leq \delta \right \}\
\end{equation}
(recall that by our assumption we have $\E f(B_1) \geq \alpha \sqrt{\log N}$).
In order to bound $\P(E_1)$, we will need the next lemma whose point is that if $|V_t|$ is at some point close to $1$, then $\EE \left [f(B_1) - f(B_t) | \mathcal{F}_t \right ]$ cannot be too large.
\begin{lemma} \label{lem111}
Let $\{ \mu_i \}_{i=1}^N$ be such that $\mu_i \leq 0$ for all $1\leq i\leq N$. Define $\tilde f: \RR^N \mapsto \R$ by (recall that $|v_i| \leq 1$)
\begin{equation} \label{deftildef}
\tilde f (x) = \sup_{1 \leq i \leq N} (x \cdot v_i + \mu_i) \mbox{ for all } x\in \RR^N\,.
\end{equation}
Also define $
\epsilon = 1 - \sup_{\theta \in \Sph} \int_{\RR^N} \langle x, \theta \rangle \tilde f(x) d \gamma(x)$.
Then we have
$$
\int_{\RR^N} \tilde f(x) d \gamma(x) \leq 10 (1 + \sqrt{\epsilon \log N} )\,, \mbox{ for all } N\in \N.
$$
in particular, one has
\begin{equation} \label{eqepsg0}
\epsilon \geq 0.
\end{equation}
\end{lemma}
\begin{proof}
Pick $\theta\in \Sph$  such that
\begin{equation}\label{eq-def-theta}
1 - \int_{\RR^N} \langle x, \theta \rangle \tilde f(x) d \gamma(x) = \epsilon\,.
\end{equation}
 For each $x \in \RR^N$, consider the unique representation $x = y + z \theta$ where $z = \langle x, \theta \rangle$ and $y\in \theta^\perp$ (i.e., $\langle y, \theta \rangle = 0$). Denote by $\gamma^1$ and $\gamma^{N-1}$ standard Gaussian measures in dimension 1 and $N-1$ respectively, we can view $\gamma^1$ as a measure on $\mathrm{span}\{\theta\}$ and $\gamma^{N-1}$ a measure on $\theta^\perp$. It is clear that if $x\sim \gamma$, we have $(z, y) \sim \gamma^1 \otimes \gamma^{N-1}$. Therefore, we get that
$$
\int_{\RR^N} \langle x, \theta \rangle \tilde f(x) d \gamma(x) = \int_{\theta^\perp} \int_{\RR} z \tilde f(y + z \theta) d \gamma^1(z)  d \gamma^{N-1} (y)\,.
$$
Applying integration by parts to $ \int_{\RR} z \tilde f(y + z \theta) d \gamma^1(z) $, we obtain that
\begin{equation}\label{eq-nabla}
\int_{\RR^N} \langle x, \theta \rangle \tilde f(x) d \gamma(x) = \int_{\theta^\perp} \int_{\RR}  \left ( \frac{\partial}{\partial z} f (y + z\theta) \right )  d \gamma^1(z)  d \gamma^{n-1} (y) = \int_{\RR^N} \langle \nabla \tilde f(x), \theta \rangle d \gamma(x)\,.
\end{equation}
For $x \in \RR^N$, write
$$
i^*(x) = \arg \max_{1 \leq i \leq N}  (x \cdot v_i + \mu_i)
$$
(note that the maximizer is unique with probability 1 when we sample $x\sim \gamma$ and thus $i^*(x)$ is well-defined almost surely. Here we use the legitimate assumption that the vectors $\{v_i\}$ are distinct). By definition of $\tilde f$, we see that $\nabla \tilde f(x) = v_{i^*(x)}$. Combined with \eqref{eq-def-theta} and \eqref{eq-nabla}, it follows that
$$
\EE \langle v_{i^*(\Gamma)}, \theta \rangle  = 1 -\epsilon\,.
$$
where $\Gamma$ is a standard Gaussian random vector in $\RR^N$. Recall that $|v_i| \leq 1$ for all $1\leq i\leq N$. In view of the last equation, this fact gives $\epsilon \geq 0$. As a consequence of Markov's inequality, this fact also teaches us that
\begin{equation}\label{eq-i*-I}
\PP \left ( \langle v_{i^*(\Gamma)}, \theta \rangle \geq 1 - 10 \epsilon \right ) \geq 9/10.
\end{equation}
Let $I \subset [N]$ be the set of indices $i$ such that $\langle v_{i}, \theta \rangle \geq 1 - 10 \epsilon$. For every $i\in I$, write $v_i = u_i \theta + \tilde v_i$ where $u_i = \langle v_i, \theta \rangle$ and $\tilde v_i \in \theta^\perp$. By our assumption on $I$, we have $|\tilde v_i| \leq \sqrt{20\epsilon}$ for all $i\in I$. Therefore, we have
\begin{align*}
\EE \sup_{i\in I} \langle \Gamma , v_i\rangle  & \leq
\EE | \langle \Gamma, \theta \rangle|  + \EE \sup_{i\in I} \langle \tilde v_i, \Gamma \rangle \leq 1 + \sqrt{40 \epsilon \log |I|}\,,
\end{align*}
where the last inequality follows from \eqref{eq-simple-upper-exp}. Combined with \eqref{eq:borel}, it then follows that (note that $|I | \leq N$)
$$
\PP \left (\sup_{i\in I} \langle \Gamma, v_i \rangle \geq \sqrt{40 \epsilon \log N } + 10 \right ) \leq  1/5.
$$
Combined with \eqref{eq-i*-I}, using a union bound we get that
$$\PP \left ( \sup_{i\in [N]} \langle \Gamma, v_i \rangle \geq \sqrt{40 \epsilon \log N } + 10 \right )  \leq  1/2\,.$$
Together with another application of \eqref{eq:borel}, it completes the proof of the lemma.
\end{proof}

The next lemma applies the above in order to show that with high probability, either $V_t$ remains bounded from $1$ for a finite inteval of time, or $f(B_t)$ becomes rather large within a short time.

\begin{lemma}\label{lem-B-A}
Let $E_1, E_2$ be defined as in \eqref{eq-def-E-12}. For $\epsilon \leq \alpha^2 \cdot 10^{-4}$ and an absolute constant $C>0$, we have
$$
\PP( E_2 \setminus E_1 ) \leq C N^{-\alpha^2/32}\,.
$$
\end{lemma}
\begin{proof}
Suppose that $E_2 \setminus E_1$ holds, and denote by
$$T = \min\{t \geq 0: V_t \geq 1 - \epsilon\}$$
to the first time in which $V_t \geq 1 - \epsilon$. By definition of $E_1^C$, we have $T \leq \delta$.
Using the decomposition that $B_1 = B_t + (B_1 - B_t)$ where $\frac{B_1 - B_t}{\sqrt{1-t}} $ has density function $\gamma$ and is independent of $B_t$, we get that
\begin{equation}\label{eq-g-t-martingale}
\int_{\RR^N} g_t(x) d \gamma(x) = \frac{\EE[f(B_1) - f(B_t) | \mathcal{F}_t]}{\sqrt{1-t}}.
\end{equation}
Consequently, we have
$$
\EE[f(B_1) | \mathcal{F}_T ] = f(B_T) + \EE[ f(B_1) - f(B_T) | \mathcal{F}_T ] = f(B_T) + \sqrt{1-t} \int_{\RR^N} g_T(x) d \gamma(x).
$$
Recalling \eqref{eq-g-t-V-t}, we see that
$$
\Big|\int_{\RR^N} x g_T(x) d \gamma(x)\Big |^2 \geq 1 - \epsilon\,.
$$
Therefore,  there exists $\theta \in \Sph$ such that
$$
\int_{\RR^N} g_T(x)\langle x, \theta \rangle d \gamma(x) \geq \sqrt{1 - \epsilon} \geq 1- \epsilon\,.
$$
We claim that Lemma \ref{lem111} can be applied with the function $g_T(x)$ (conditioning on the filtration $\mathcal F_T$) used in place of the function $\tilde f$. Indeed, since $f(cx) = cf(x)$ for all $x\in \RR^N$ and $c > 0$, we get that
\begin{align*}
g_T(x)
&=  f (x + B_T / \sqrt{1-T}) - f(B_T / \sqrt{1-T})\\
& =
\sup_{1 \leq i \leq N} \langle x + B_T / \sqrt{1-T}, v_i \rangle - \sup_{1 \leq i \leq N} \langle B_T/ \sqrt{1-T}, v_i \rangle \\
& =
\sup_{1 \leq i \leq N} \left ( x \cdot v_i + \frac{B_T}{\sqrt{1-T}} \cdot v_i - \sup_{1 \leq i \leq N} \frac{B_T}{\sqrt{1-T}} \cdot v_i \right )\,.
\end{align*}
This implies that it admits the form \eqref{deftildef}. Applying Lemma~\ref{lem111}, and using the assumption $\epsilon \leq \alpha^2 \cdot 10^{-4}$, we get
$$\int_{\RR^N} g_T(x) d \gamma(x)  \leq 10 (\sqrt{\epsilon \log N} + 1) \leq  \frac{\alpha}{10} \sqrt{\log N} + 10\,.$$
This implies that on the event $E = \{T \leq \delta\} \cap \{f(B_t) \leq \alpha \sqrt{\log N}/2, \forall 0\leq t\leq T\}$ (note that $E\supseteq E_2 \setminus E_1$), we have
$$
\EE[f(B_1) | \mathcal{F}_T ] \leq \frac{3\alpha}{4} \sqrt{\log N} + 10\,.
$$
Applying \eqref{eq:borel} (the non-centered version) to
$$f(B_1) = \sup_{1\leq i\leq N} \{\langle v_i, B_1 - B_T\rangle +  \langle v_i, B_T \rangle\}\,,$$
(where we treat $B_T$ as deterministic numbers as we conditioned on $\mathcal F_T$) we obtain that
\begin{equation}
\label{eq-f-B-1}
\left . \PP \left (f(B_1) \leq \frac{3\alpha}{4} \sqrt{\log N} + 20 ~ \right \vert E \right ) \geq 1/2\,.
\end{equation}
Recall the definition of $\alpha$, according to which
$$
\EE[ f(B_1) ] \geq \alpha \sqrt{\log N}.
$$
Another application of \eqref{eq:borel} gives that
$$\PP \left  (f(B_1) \leq \frac{3\alpha}{4} \sqrt{\log N} + 20 \right )\leq C N^{-\alpha^2/32}\,,$$
where $C>0$ is an absolute constant.  Combined with \eqref{eq-f-B-1}, we see that
\begin{equation*}
\P(E_2 \setminus E_1) \leq \P(E) \leq 2CN^{-\alpha^2/32}\,. \qedhere
\end{equation*}
\end{proof}

\begin{proof}[Proof of Theorem~\ref{thm-deviation}]
We first bound $\P(E_2)$ from below, and we will employ the idea from reflection principle of Brownian motion. Defining
$$T' = \min\{t: f(B_t) \geq \alpha \sqrt{\log N}/2\}\,,$$
we see that $E_2 = \{T' > \delta\}$. Let us denote by $i^*_{T'}$ the maximizer of $f(B_{T'})$. That is, $f(B_{T'}) = \langle v_{i^*_{T'}}, B_{T'} \rangle$. Then we have, on the event $T'\leq \delta$,
$$f(B_\delta) \geq f(B_{T'}) + \langle v_{i^*_{T'}}, B_\delta - B_{T'} \rangle\,.$$
Observe that whenever the event $T' \leq \delta$ holds, then $(B_\delta - B_{T'})$ has a origin-symmetric distribution conditioned on $\mathcal{F}_{T'}$. We infer that
$$\P(f(B_\delta) \geq \alpha \sqrt{\log N} /2) \geq \P(T'\leq \delta) /2\,.$$
Combined with an application of \eqref{eq:borel}, it follows that
$$\P(E_2^c) = \P(T' \leq \delta) \leq 4 N^{-\alpha^2}\,,$$
where we choose $\delta = 1/100$.
Choosing $\epsilon = 10^{-4} \alpha^2$, it follows from an application of Lemma~\ref{lem-B-A} that \begin{equation} \label{PE1big}
\P(E_1^c) \leq C' N^{-\alpha^2/32}
\end{equation}
for an absolute constant $C'>0$.

In order to complete the proof, note that $S_t - \E S_1$ is a mean-zero continuous-time martingale, so according to the Dambis / Dubins-Schwartz theorem, there exists standard a Brownian motion $\{W_t\}_{t \geq 0}$ such that
$$
W_{[S]_t} = S_t, ~~ \forall 0 \leq t \leq 1.
$$
An elementary fact about the one-dimensional Brownian motion is that
\begin{equation}
\PP \left ( \max_{0 \leq t \leq \tau}  |W_t| \geq s \right ) \leq 4 \mathrm{e}^{- \frac{s^2}{2 \tau}}, ~~ \forall \tau, S \geq 0
\end{equation}
As a consequence of equation \eqref{eqepsg0} we know that $V_t \leq 1$ for all $0\leq t\leq 1$. Therefore, on $E_1$ we have $[S]_1 \leq  1- \epsilon \delta \leq 1- 10^{-6} \alpha^2$, and combined with the last inequality,
$$
\P \Bigl ( \bigl \{|S_1 - E S_1| \geq \beta \sqrt{\log N} \bigr \} \cap E_1 \Bigr ) \leq \P \left (\max_{0\leq t\leq 1- 10^{-6} \alpha^2} |W_t| \geq \beta \sqrt{\log N}  \right ) \leq 4 N^{-\beta^2/2(1-10^{-6}\alpha^2)}.
$$
Combining with \eqref{PE1big} and using a union bound finally gives
\begin{align*}
\P(|S_1 - E S_1| \geq \beta \sqrt{\log N}) \leq
 4 N^{-\beta^2/2(1-10^{-6}\alpha^2)} + 2C N^{-\alpha^2/32}\,.
\end{align*}
This completes the proof of Theorem~\ref{thm-deviation}.
\end{proof}

\section{A lower bound for standard deviation in terms of expectation} \label{SecDevExp}

In this section, we provide a proof for Theorem~\ref{thm-var-exp}. As usual, for all $i\in [N]$ we associate a unit vector $v_i\in \RR^N$ such that we can represent $X_i = \langle v_i, \Gamma \rangle$ for all $i\in [N]$, where $\Gamma\in\RR^N$ is a standard Gaussian vector. We define a convex body,
$$
K = \bigcap_{i\in [N]} \{x ; \langle x, v_i \rangle \leq 1  \}\,.
$$
By slight abuse of notation, we will allow ourselves to denote by $\gamma(\cdot)$ the \emph{density of} the standard Gaussian measure in $\RR^N$. For any (smooth enough) set $A \subseteq \RR^N$ whose Hausdorff dimension is $n-1$, we define $\gamma^+(A)$ to be the Gaussian surface area of $A$, namely
$$
\gamma^+(A) = \int_{ A} \gamma(x) d \mathcal{H}_{N-1}(x)
$$
where $\mathcal{H}_{N-1}$ is the $(N-1)$-dimensional Hausdorff measure.
Further, we define
$$
L(K) = \sup_{t \geq 0} \gamma^+(K_t)
$$
where $K_t := \partial (t K)$.

The idea of the proof of the theorem will be to establish a connection between the quantities $L(K)$ and $\sqrt{\var[\sup_{i\in [N]} X_i ]}$ using the co-area formula and then to bound the quantity $L(K)$ using a geometric idea.  Recall that according to the co-area formula, for every function $\varphi(x) \in L_1(\RR^N)$ and for every Lipschitz function $u: \RR^N \to \RR$ one has
$$
\int_{\RR^n} |\nabla u(x)| \varphi(x) dx = \int_{-\infty}^{\infty} \int_{u^{-1}(t)} \varphi(x) d \mathcal{H}_{N-1}(x)  dt.
$$
Define $u(x) = \sup_{i\in [N]} \langle x, v_i\rangle $. The assumption that $\var[X_i] = 1$ for all $1 \leq i \leq n$ implies that $|v_i| = 1$, which in turn implies that $|\nabla u| = 1$ almost everywhere. Moreover, by definition we have
\begin{equation}\label{eq-GSA}
u^{-1}(t) = K_t.
\end{equation}
Therefore, for all $0 \leq a < b$, we can define
$$
\varphi(x) = \one_{ u(x) \in [a,b] } \gamma (x)
$$
and according to the co-area formula
\begin{align*}
\PP(u(\Gamma) \in [a,b]) &= \int_{ u(x) \in [a,b] } d \gamma(x) = \int_a^b \int_{K_t} \gamma(x) d \mathcal{H}_{N-1}(x) dt \\
&=  \int_a^b \gamma^+( K_t ) dt \leq (b-a) L(K).
\end{align*}
By taking $a = \EE [u(\Gamma)] - 2 \sqrt{\var [u(\Gamma)]}$ and
$b = \EE [u(\Gamma)] + 2 \sqrt{\var [u(\Gamma)]}$, and using Chebyshev's inequality, we finally have
$$
\frac{3}{4} \leq \PP(u(\Gamma) \in [a,b]) \leq 4 \sqrt{\var [u(\Gamma)]} L(K)
$$
or, in other words,
\begin{equation} \label{coarea}
\sqrt{\var \sup_{i\in [N]} X_i}  \geq \frac{1}{6} L(K)^{-1}.
\end{equation}
The next lemma provides an upper bound for $L(K)$ in terms of the expected supremum.
\begin{lemma}\label{lem-GSA}
With definitions above, there exists an absolute constant $C>0$ such that
$$
L(K) < C \E [ \sup_{i\in [N]} X_i ].
$$
\end{lemma}
\begin{proof}
For all $1 \leq i \leq N$, let $F_i$ be the $N-1$-dimensional facet of $K$ corresponding to the constraint $\langle x, v_i \rangle \leq 1$, in other words
$$
F_i = \{x ; ~ \langle x, v_i \rangle = 1 \mbox { and } \langle x, v_j \rangle < 1 \mbox { for all } j \neq i  \}.
$$
Following an idea of Nazarov, we also define the sets
$$
\tilde F_i = \{x + s v_i; ~ x \in F_i, ~ s > 0 \}.
$$
It is not hard to verify that the sets $\tilde F_i$ are disjoint (the reader is advised to draw a picture). Next, the assumption that for all $1 \leq i \leq N$ one has $\var[X_i] = 1$ implies that $|v_i|=1$. By considering the unique decomposition $x = s v_i + y$ where $s \in \RR$ and $y \in v_i^\perp$ and since the $\gamma = \gamma^1 \otimes \gamma^{N-1}$ where $\gamma_1$ and $\gamma^{N-1}$ are the one-dimensional and $(N-1)$-dimensional Gaussian measures respectively, we have for all $i$ and for all $t > 0$,
\begin{align*}
\gamma( t \tilde F_i) = \int_{t P_{v_i^\perp} F_i} \int_t^\infty \gamma^1(ds) \gamma^{N-1} (dy) = \gamma^{N-1} (t P_{v_i^\perp} F_i ) \gamma^1 \bigl ( [t,\infty ) \bigr )
\end{align*}
where $P_{v_i^\perp}$ denotes the orthogonal projection onto $v_i^\perp$. Moreover, by definition of the measure $\gamma^+$ and since $|v_i| = 1$, we have
$$
\gamma^+( t F_i ) = \frac{d \gamma^1}{dx} (t) \gamma^{N-1} ( t P_{v_i^\perp} F_i )
$$
where $\frac{d \gamma^1}{dx}$ denotes the density of the one-dimensional standard Gaussian measure. Combining the two above inequalities and using the observation that the sets $F_i$ and $\tilde F_i$ are disjoint we get that
$$
\frac{\gamma \left (t \bigcup_{1 \leq i \leq N} \tilde F_i \right )} { \gamma^+( K_t)} = \frac{ \int_t^\infty e^{-s^2/2} ds }{ e^{-t^2/2} } \geq \frac{c}{t}
$$
where $c>0$ is a universal constant. Using the fact that $\gamma \left (t \bigcup_{1 \leq i \leq N} \tilde F_i \right ) \leq 1$, it follows that
$$\gamma^+(K_t) \leq c^{-1} t \mbox{ for all } t\geq 0\,.$$

By definition of $L(K)$ and in view of the above formula, in order to prove the lemma it remains to show that $\gamma^+(K_t) < C \E\sup_{i\in [N]} X_i$ for all $t> 2 \E\sup_{i\in [N]} X_i$ and for some universal constant $C>0$. For convenience, define
$$m = \E \sup_{i \in [N]} X_i.$$
By Markov's inequality, we have $\gamma(2mK) \geq 1/2$. Using the result of \cite{Borell75b} on the log-concavity of Gaussian measure, we see that $\gamma(tK)$ is a log-concave function in $t$.  This implies that $\frac{\gamma^+(K_t)}{\gamma(t K)}$ is decreasing in $t$.  Thus, for all $t\geq 2m$ we have
$$\gamma^+(K_t) \leq \gamma(t K) \cdot \frac{\gamma^+ (K_{2m})}{\gamma(2m K)} \leq 2 \gamma^+(K_{2m}) \leq 4 c^{-1}m.$$
This completes the proof of the lemma.
\end{proof}

\begin{proof}[Proof of Theorem~\ref{thm-var-exp}]
Combine equation \eqref{coarea} with Lemma \ref{lem-GSA}.
\end{proof}

\section*{Acknowledgement}

We thank  Antonio Auffinger, Wei-Kuo Chen, Galyna Livshyts, Michel Ledoux, Elchanan Mossel  and Ofer Zeitouni for helpful discussions.  This work was initiated when J.D. and A.Z. were visiting theory group of Microsoft Research at Redmond. We thank MSR for the hospitality. 

\def\cprime{$'$}

\end{document}